\documentclass[11pt, a4paper]{article}
\usepackage{amsmath,amssymb,latexsym,graphics,epsfig}
\usepackage{hyperref}
\usepackage{color}
\usepackage{amsthm}
\usepackage{graphicx,url}
\usepackage{amsopn}

\newtheorem{theo}{Theorem}
\newtheorem{lema}[theo]{Lemma}

\DeclareMathOperator{\dist}{dist}
\DeclareMathOperator{\tr}{tr}

\DeclareMathOperator{\spec}{sp}

\def\j{\mbox{\boldmath $j$}}

\def\r{\mbox{\boldmath $r$}}

\def\vec0{\mbox{\boldmath $0$}}

\def\A{\mbox{\boldmath $A$}}

\def\F{\mbox{\boldmath $F$}}
\def\I{\mbox{\boldmath $I$}}
\def\J{\mbox{\boldmath $J$}}
\def\K{\mbox{\boldmath $K$}}
\def\L{\mbox{\boldmath $L$}}
\def\M{\mbox{\boldmath $M$}}
\def\N{\mbox{\boldmath $N$}}

\def\U{\mbox{\boldmath $U$}}

\def\G{\Gamma}
\def\Re{\mathbb R}

\begin{document}
\title{The Laplacian Spectral Excess \\ Theorem  for Distance-Regular Graphs\footnote{This version is published in
Linear Algebra and its Applications 458 (2014), 245--250.}
}

\author{E.R. van Dam$^a$, M.A. Fiol$^b$
\\ \\
{\small $^a$Tilburg University, Dept. of Econometrics and O.R.} \\
{\small  Tilburg, The Netherlands}\\
{\small (e-mail: {\tt Edwin.vanDam@uvt.nl)}} \\
{\small $^b$Universitat Polit\`ecnica de Catalunya, BarcelonaTech} \\
{\small Dept. de Matem\`atica Aplicada IV, Barcelona, Catalonia}\\
{\small (e-mail: {\tt
fiol@ma4.upc.edu})} \\
 }

\date{}
\maketitle

\begin{abstract}
The spectral excess theorem states that, in a regular graph $\G$, the average excess, which is the mean of the numbers of vertices at maximum distance from a vertex, is bounded above by the spectral excess (a number that is computed by using the adjacency spectrum of $\G$), and $\G$ is distance-regular if and only if equality holds. In this note we prove the corresponding result by using the Laplacian spectrum without requiring regularity of $\G$.
 \end{abstract}

 \noindent{\em Keywords:} Distance-regular graph; Spectral excess theorem; Laplacian spectrum; Orthogonal polynomials.

\section{Introduction}
The spectral excess of a regular (connected) graph $\G$ is a number which can be computed from its
(adjacency matrix) spectrum, whereas its average excess is the mean of the numbers of vertices at
maximum distance from a vertex. The spectral excess theorem, due to Fiol and Garriga \cite{fg97} states that $\G$ is distance-regular
if and only if its spectral excess equals its average excess (see Van Dam \cite{vd08} and Fiol, Garriga, and Gago \cite{fgg10} for short proofs). Since the
paper \cite{fg97} appeared, some attempts have been made to prove a version  of the spectral excess theorem that does not require regularity of $\G$ (see Lee and Weng \cite{lw11,lw14} and Fiol \cite{f13}). The problem with these attempts is that the  obtained equalities only lead to distance-regularity in some specific cases (graphs with extremal diameter, bipartite graphs, etc.), some of them already covered by the results in \cite{fg97}.

In this note we show that the right approach to the spectral excess theorem for general graphs is to derive it from the Laplacian spectrum of the graph. This approach was motivated by the fact that a bound on the excess in terms of the Laplacian eigenvalues by the first author \cite[Thm.~3.1]{vd98} equals an expression for the excess in strongly distance-regular graphs by the second author and Garriga \cite[Thm.~3.3]{fg02}, \cite[Cor.~2.5]{f00}. In the following section we will recall the basic terminology and earlier results. Then the main result is derived in the last section.

\section{Preliminaries}

Let us first recall some basic notation and results on which our study is based. For more background on spectra of graphs,
distance-regular graphs, and their characterizations, see
\cite{b93,bcn89,bh12,cds82,dkt12,g93}.

Throughout this paper, $\G$ denotes a (finite, simple, and connected) graph
with vertex set $V$, order $n=|V|$, and diameter $D$. Its $(0,1)$-adjacency matrix is denoted by $\A$.
The set of
vertices at distance $i$ from a given vertex $u\in V$ is
denoted by $\G_i(u)$, for $i=0,1,\dots,D$, and $k_i(u)=|\G_i(u)|$. We abbreviate $k_1(u)$ by $k(u)$, the degree of vertex $u$. Also, the closed $i$-neighborhood of $u$ is $N_i(u)=\G_0(u)\cup\cdots\cup\G_i(u)$. Recall that, for every $i=0,1,\ldots,D$, the distance matrix $\A_i$ has entries $(\A_i)_{uv}=1$ if $\dist(u,v)=i$, and $(\A_i)_{uv}=0$ otherwise. In particular, $\A_0=\I$ and $\A_1=\A$.
Then, it is well-known that $\G$ is distance-regular if and only if there exist so-called distance polynomials $p_0,\ldots,p_D$, with $\deg p_i=i$, such that $p_i(\A)=\A_i$ for every $i=0,\ldots,D$.

The Laplacian matrix of $\G$ is the matrix $\L=\K-\A$, where $\K$ is the diagonal matrix with entries $K_{uu}=k(u)$,
for $u\in V$. The  (Laplacian) spectrum of $\G$  is
$\spec \G=\spec \L= \{\theta_0(=0)^{m_0},\theta_1^{m_1},\ldots,\theta_d^{m_d}\}$,
where  $\theta_0=0<\theta_1<\cdots <\theta_d$ are the distinct eigenvalues, and the superscripts
stand for their multiplicities $m_i=m(\theta_i)$.
In particular, since $\G$ is
connected, $m_0=1$,  and  $\theta_0$ has eigenvector $\j$, the all-$1$ vector.
We emphasize that throughout this note, $d$ will always denote the number of distinct eigenvalues minus one, and $D$ will denote the diameter.
Let $\F_i$, $i=0,1,\ldots,d$, be the idempotents of $\L$, that is $\F_i=\frac{1}{\phi_i}\prod_{j\neq i}(\L-\theta_j\I)=\U_i\U_i^{\top}$, where $\phi_i=\prod_{j\neq i}(\theta_i-\theta_j)$, and $\U_i$ is an $n\times m_i$ matrix having orthonormal eigenvectors of $\theta_i$ as columns. In particular, $\F_0=\frac{1}{n}\J$, with $\J$ being the all-$1$ matrix.

\subsection*{Laplacian predistance and Hoffman polynomials}
Given a graph $\G$ with spectrum as above,
the {\em Laplacian predistance polynomials}
$r_0,\ldots,r_d$, introduced analogously in \cite{fg97} for the adjacency spectrum,
are the orthogonal polynomials with respect to the scalar product
\begin{equation}\label{scalar-prod-glob}
\langle p,q\rangle_{\G}=\frac{1}{n}\tr (p(\L)q(\L))= \frac{1}{n}\sum_{i=0}^d m_i p(\theta_i)q(\theta_i), \qquad p,q\in \Re_{d}[x],
\end{equation}
normalized in such a way that
$\|r_i\|_{\G}^2=r_i(0)$. (This makes sense since it is known that, for any sequence of such orthogonal polynomials $p_0,\ldots,p_d$, we always have
 $p_i(0)\neq 0$.)
 As every sequence of orthogonal polynomials, the $r_i$s satisfy a three-term recurrence of the form
 \begin{equation}
 \label{recur}
 xr_i=\beta_{i-1}r_{i-1}+\alpha_i r_i+\gamma_{i+1}r_{i+1},\qquad i=0,\ldots,d,
 \end{equation}
 where $\beta_{-1}=\gamma_{d+1}=0$, and $\beta_{i-1}\gamma_i>0$ for $i=1,\ldots,d$. In fact, in our case it can be proved that the betas and gammas are negative, in a similar way as in \cite[Lemma 2.3]{adf14}.
 Also, similar as in the case of the adjacency predistance polynomials, it can be proved that $\alpha_i+\beta_i+\gamma_i=\theta_0=0$, $i=0\ldots,d$, and
$r_d(0)=n\left(\sum_{i=0}^d\frac{\phi_0^2}{m_i\phi_i^2}\right)^{-1}$, see \cite{cffg09}.

Here we can also consider a {\it Hoffman-like  polynomial} (see \cite{hof63} for the case of the adjacency spectrum), defined as
 $H=\frac{n}{\phi_0}\prod_{i=1}^d (x-\theta_i)$, where we recall that $\phi_0=\prod_{i=1}^d (-\theta_i)$. This polynomial
satisfies $H(\L)=n\F_0=\J$ (independently of whether $\G$ is regular or not), and
$H= r_0+r_1+\cdots +r_d$.  The latter follows from the fact that $\langle H,r_i\rangle_{\G}=\frac1n \tr (H(\L)r_i(\L)) =\frac 1n \tr (r_i(0)J) =\|r_i\|_{\G}^2$ for every $i=0,\ldots,d$. From $H(\L)=\J$ it follows that the diameter $D$ is at most $d$.

\section{The Laplacian spectral excess theorem}
In this section we prove the main result, which can be considered as the spectral excess theorem for nonregular graphs.
As in the short proofs of the (standard) spectral excess theorem,  we prove the Laplacian version of such a result in two steps, that correspond to the lemmas below.
Although the proofs of such lemmas are basically the same as in \cite{fgg10}, we have detailed them in order to have this note more self-contained.

\begin{lema}
\label{lema1}
Let $\G$ be a graph with Laplacian matrix $\L$, predistance polynomials $r_0,\ldots,r_d$,
and distance matrices $\A_i$, $i=0,\ldots,d$. If $r_d(\L)=\A_d$ then, $r_i(\L)=\A_i$ for every $i=0,1,\ldots,d$.
\end{lema}
\begin{proof}
We only show the case $i=d-1$, as the other cases are proved analogously. From the hypothesis and $H(\L)=\J=\sum_{i=0}^d \A_i$, we get that $\r_0(\L)+\cdots+r_{d-1}(\L)=\A_0+\cdots+\A_{d-1}$. We then distinguish three cases:
\begin{itemize}
\item
If $\dist(u,v)=d$, we clearly have $(r_{d-1}(\L))_{uv}=0$.
\item
If $\dist(u,v)=d-1$, the above gives $(r_{d-1}(\L))_{uv}=1$.
\item
If $\dist(u,v)\le d-2$, the three-term recurrence for $i=d$ is $xr_d=\beta_{d-1}r_{d-1}+\alpha_d r_d$. Then, when applied to $\L$, we get that
$\beta_{d-1}r_{d-1}(\L)=\L \A_d-\alpha_d\A_d$.  But $(\L \A_d)_{uv}=\sum_{w\in V}(\L)_{uw}(\A_d)_{wv}= \sum_{w\in N_1(u)}(\L)_{uw}(\A_d)_{wv}=0$ since
$\dist(w,v)\le \dist(u,v)+1\le d-1$. Thus, $(r_{d-1}(\L))_{uv}=0$ since $\beta_{d-1}\neq 0$.
\end{itemize}
Consequently,  $r_{d-1}(\L)=\A_{d-1}$.
\end{proof}

\begin{lema}
\label{lema2}
Let $\G$ be a graph with Laplacian predistance polynomial $r_d$. Let $\overline{k}_d$ be the average
over $V$ of the numbers $k_d(u)=|\G_d(u)|$. Then,
$$
\overline{k}_d\le r_d(0)
$$
and, in case of equality,  $r_d(\L)=\A_d$.
\end{lema}
\begin{proof}
First, notice that
$
\langle r_d(\L),\A_d\rangle=\langle H(\L),\A_d\rangle=\langle \J,\A_d\rangle=\|\A_d\|^2=\overline{k}_d
$. Note that we use the inner product on matrices defined by $\langle \M,\N\rangle =\frac 1n \tr (\M\N)$, so that $\langle p,q\rangle_{\G}=\langle p(\L),q(\L)\rangle$ by \eqref{scalar-prod-glob}. Also, by the Cauchy-Schwarz inequality,
$|\langle r_d(\L),\A_d\rangle|^2\le \|r_d\|^2_{\G}\|\A_d\|^2=r_d(0)\overline{k}_d$. Combining the above, the inequality holds. Moreover, in case of equality, $r_d(\L)=c\A_d$ for some constant $c$. Finally, we have that $c=1$ because $\overline{k}_d =\langle r_d(\L),\A_d\rangle=\langle c\A_d,\A_d\rangle=c \overline{k}_d$ (and $\overline{k}_d=r_d(0)>0$).
\end{proof}

Now we are ready to give the spectral excess theorem for general graphs or, what we could call,
the Laplacian spectral excess theorem.

\begin{theo}
\label{teo(basic)}
Let $\G$ be a graph on $n$ vertices, with Laplacian spectrum $\{\theta_0(=0)^{m_0=1}, \linebreak \theta_1^{m_1}, \ldots, \theta_d^{m_d}\}$, and Laplacian predistance polynomial $r_d$. Let $\overline{k}_d$ be the average over $V$ of the numbers $k_d(u)=|\G_d(u)|$. Then, $\G$ is distance-regular
if and only if
$$
\overline{k}_d= r_d(0)=n\left(\sum_{i=0}^d\frac{\phi_0^2}{m_i\phi_i^2}\right)^{-1},
$$
where $\phi_i=\prod_{j\neq i} (\theta_i-\theta_j)$, $i=0,\ldots,d$.
\end{theo}

\begin{proof}
For sufficiency, Lemmas \ref{lema1} and \ref{lema2} imply that $r_i(\L)=\A_i$ for every $i=0,1\ldots,d$.
In particular, for $i=1$, there exist some constants
$\omega_1\neq 0$ and $\omega_2$ such that $\omega_1\L+\omega_2\I=\A$,
which implies that $(\L)_{uu}=-\omega_2/\omega_1$ for every $u\in V$.
Then, $\G$ is regular with degree $k=-\omega_2/\omega_1$, and $\L=k\I-\A$.
In turn, this assures the existence of the distance polynomials $p_0,\ldots,p_d$ of $\G$,
just take $p_i(x)=r_i(k-x)$ for $i=0,\ldots,d$, and hence $\G$ is distance-regular (with $D=d$).
 Necessity follows straightforwardly from $r_d(x)=p_d(k-x)$.
\end{proof}

Let us illustrate this Laplacian spectral excess theorem in the case of graphs with three
 Laplacian eigenvalues, that is, the case $d=2$. Such graphs have been studied in \cite{mumubar}.

Note that for every $d$, we have that $r_0=1$, $r_1=\frac1{\gamma_1}(x-\alpha_0)$,
and that $\alpha_0=\frac1n \tr \L= \overline{k}$, the average degree.
Moreover, it can be shown that
$\gamma_1=-1+\overline{k}-\overline{k^2}/\overline{k}$, where $\overline{k^2}=\frac 1n \sum_{u \in V}k(u)^2$,
using among others that $\frac 1n \tr \L^2 = \overline{k^2}+\overline{k}$.
Note that for a $k$-regular graph we thus have that $\alpha_0=k$, and $\gamma_1=-1$,
so that $r_1=k-x$, which corresponds to the fact that $\A=k\I-\L$.

For the case $d=2$, the inequality $\overline{k}_2 \leq r_2(0)$ of Lemma \ref{lema2}
can be rewritten as $n-1-\overline{k} \leq H(0)-r_0(0)-r_1(0)=n-1+\frac{\alpha_0}{\gamma_1}$,
which is equivalent to the inequality $\gamma_1 \leq -1$ (recall that $\gamma_1$ is negative),
which in turn is equivalent to the inequality $\overline{k^2} \geq \overline{k}^2$.
This is of course a standard inequality, and equality holds precisely when the graph is regular.
Thus we may draw the (known) conclusion that a graph with three Laplacian eigenvalues
is distance-regular (strongly regular in fact) precisely when it is regular.

A (non-regular) example with $D=d=3$ is given by the path on four vertices, which has Laplacian spectrum $\{0,2-\sqrt{2},2,2+\sqrt{2}\}$. The betas, alphas, and gammas are as in below table.

\begin{table}[h]\begin{center}
\begin{tabular}{|c||c|c|c|c|}
  \hline
  $i$        & 0 & 1 & 2     & 3      \\
  \hline
  \hline
  $\beta_i$  & $-3/2$ & $-16/21$ & $-7/10$ &    \\
  \hline
  $\alpha_i$ & $3/2$ & $27/14$ & $62/35$ & $4/5$ \\
  \hline
  $\gamma_i$ &   & $-7/6$ & $-15/14$ & $-4/5$  \\
  \hline
\end{tabular}\end{center}
\end{table}

The Laplacian predistance polynomials are
\begin{align*}
&r_0=1,\\
&r_1=-\frac 67 x + \frac 97,\\
&r_2=\frac45 x^2 - \frac{96}{35} x + \frac{32}{35},\\
&r_3=-x^3 + \frac{26}5x^2 - \frac{32}5x + \frac45.
\end{align*}
Consequently, Lemma \ref{lema2} gives the inequality $\overline{k}_3 \leq \frac45$.
Indeed, in this graph, we have that $\overline{k}_3 = \frac12$. Note that this example has constant $k_2=1$,
which reminds us of the version of the spectral excess theorem for regular graphs with $d=3$ in \cite{vdh97}
in terms of the number of vertices at distance two.

\noindent{\large \bf Acknowledgments.} The authors thank a referee for comments on an earlier version.
This work was done while the second author was visiting the Department of Econometrics and Operations Research, in
Tilburg University (The Netherlands).

Research supported by the
{\em Ministerio de Ciencia e Innovaci\'on}, Spain, and the {\em European Regional
Development Fund} under project MTM2011-28800-C02-01, and the {\em Catalan Research
Council} under project 2009SGR1387 (M.A.F.).



\begin{thebibliography}{99}

\bibitem{adf14}
A. Abiad, E.R. van Dam, and M.A. Fiol, Some spectral and quasi-spectral characterizations of distance-regular graphs, preprint (2014); arXiv:\href{http://arxiv.org/abs/1404.3973}{1404.3973}.

\bibitem{b93}
N. Biggs, \emph{Algebraic Graph Theory}, Cambridge University Press,
Cambridge, 1974, second edition, 1993.

\bibitem{bcn89}
A.E. Brouwer, A.M. Cohen, and A. Neumaier, \emph{Distance-Regular Graphs},
Springer-Verlag, Berlin-New York, 1989.

\bibitem{bh12}
A.E. Brouwer and W.H. Haemers, \emph{Spectra of Graphs},
Springer,
2012; available online at \url{http://homepages.cwi.nl/~aeb/math/ipm/}.

\bibitem{cffg09}
M. C\'amara, J. F\`abrega, M.A. Fiol, and E. Garriga,
Some families of orthogonal polynomials of a discrete variable and
their applications to graphs and codes, {\em Electron. J. Combin.} 16(1) (2009), \#R83.

\bibitem{cds82}
D.M. Cvetkovi\'c, M. Doob and H. Sachs, {\em Spectra of Graphs. Theory and Application},
VEB Deutscher Verlag der Wissenschaften, Berlin, second edition, 1982.

\bibitem{vd98}
E.R. van Dam, Bounds on special subsets in graphs, eigenvalues and association schemes, {\em J. Algebraic Combin.} 7 (1998), 321--332.

\bibitem{vd08}
E.R. van Dam, The spectral excess theorem for distance-regular
graphs: a global (over)view, {\em Electron. J. Combin.} 15(1)
(2008), \#R129.

\bibitem{vdh97}
E.R. van Dam and W.H. Haemers,  A characterization of distance-regular graphs with diameter three, {\em J. Algebraic Combin.} 6 (1997), 299--303.


\bibitem{mumubar}
E.R. van Dam and W.H. Haemers, Graphs with constant $\mu$ and $\overline{\mu}$, {\em Discrete Math.} 182 (1998), 293--307.

\bibitem{dkt12}
E.R. van Dam, J.H. Koolen, and H. Tanaka, Distance-regular
graphs, manuscript (2014), available online at
\url{https://sites.google.com/site/edwinrvandam/home/papers/drg.pdf}.

\bibitem{f00}
M.A. Fiol,
A quasi-spectral characterization
of strongly distance-regular graphs, {\em Electron. J. Combin.} 7
(2000), \#R51.

\bibitem{f13}
M.A. Fiol, On some approaches to the spectral excess
theorem for nonregular graphs,
\emph{J. Combin. Theory Ser. A} 120 (2013), 1285--1290.

\bibitem{fgg10}
M.A. Fiol, S. Gago, and E. Garriga,
A simple proof of the spectral excess theorem for distance-regular graphs,
{\em Linear Algebra Appl.} 432 (2010), 2418--2422.

\bibitem{fg97}
M.A. Fiol and E. Garriga,
From local adjacency polynomials to locally pseudo-distance-regular graphs,
\emph{J. Combin. Theory Ser. B} 71 (1997), 162--183.

\bibitem{fg02}
M.A. Fiol and E. Garriga, Pseudo-strong regularity around a set, {\it Linear
Multilinear Algebra} 50 (2002), 33--47.


\bibitem{g93}
C.D. Godsil, {\it Algebraic Combinatorics}, Chapman and Hall, NewYork, 1993.

\bibitem{hof63}
A.J. Hoffman, On the polynomial of a graph,
{\it Amer. Math. Monthly} 70 (1963), 30--36.

\bibitem{lw11}
G.-S. Lee, C.-W. Weng,
The spectral excess theorem for general graphs,
{\em J. Combin. Theory, Ser. A} 119 (2012), 1427--1431.

\bibitem{lw14}
G.-S. Lee, C.-W. Weng,
A characterization of bipartite distance-regular graphs,
{\em Linear Algebra Appl.} 446 (2014),  91--103.

\end{thebibliography}
\end{document}